\documentclass[a4paper, leqno]{amsart}
\usepackage[T1]{fontenc}
\usepackage[english]{babel}
\usepackage{amsmath}
\usepackage{amsthm}
\usepackage{xfrac}
\usepackage{amssymb, enumerate}
\usepackage{amsfonts}
\usepackage{dsfont}
\usepackage{graphicx}
\usepackage[lofdepth,lotdepth]{subfig}
\usepackage{mathtools,mathrsfs} 
\usepackage{cases} 
\usepackage{tikz}

\usepackage[shortlabels]{enumitem}

\usepackage{hyperref} 
\usepackage[capitalise, nameinlink, noabbrev]{cleveref} 
\hypersetup{colorlinks=true,
	linkcolor=black,
	citecolor=black,
}

\usepackage{amsbsy,upref,color,amscd}
\usepackage[active]{srcltx}
\usepackage{bbm}
\usepackage{algorithm}
\usepackage{algorithmicx}
\usepackage{algpseudocode}
\usetikzlibrary{shapes,snakes,patterns}
\usetikzlibrary{decorations.text}
\usetikzlibrary{decorations.pathmorphing, decorations.text}
\usepackage{subfig}
\usepackage{eurosym}
\usepackage{fancyhdr}
\usepackage{colorspace}

\usepackage{environ}
\makeatletter
\newsavebox{\measure@tikzpicture}
\NewEnviron{scaletikzpicturetowidth}[1]{%
  \def\tikz@width{#1}%
  \begin{lrbox}{\measure@tikzpicture}%
  \BODY
  \end{lrbox}%
  \pgfmathparse{#1/\wd\measure@tikzpicture}%
  \BODY
}
\makeatother
\usepackage{sansmath}
\usetikzlibrary{shadings,intersections,patterns,calc,angles}

\usepackage{rotating,multirow,booktabs,float}
\usepackage{pgfplots}

\pgfplotsset{compat=1.10}
\usepgfplotslibrary{fillbetween}
\usepgfplotslibrary{polar}
\usepackage{tkz-euclide}
\usepackage{environ}

\usetikzlibrary{arrows,automata}
\usetikzlibrary{positioning}
\usetikzlibrary{calc,through}

\tikzset{
    state/.style={
           rectangle,
           rounded corners,
           draw=black, very thick,
           minimum height=2em,
           inner sep=2pt,
           text centered,
           },
}

\usepackage{pgf}
\usepackage{xcolor}

\usepackage{cfr-lm}

\def\ds{\displaystyle}
\def\eps{{\varepsilon}}

\def\R{\mathbb{R}}

\definecolor{racing}{rgb}{0.7,0.1,0.2}
\definecolor{french}{rgb}{0,0.2,0.7}

\newcommand{\mean}[1]{\,-\hskip-1.08em\int_{#1}} 

\def\XXint#1#2#3{{\setbox0=\hbox{$#1{#2#3}{\int}$ }
\vcenter{\hbox{$#2#3$ }}\kern-.6\wd0}}

\DeclareMathOperator*{\osc}{osc}

\newtheorem{proposition}{Proposition}[section]
\newtheorem{theorem}[proposition]{Theorem}

\newtheorem{lemma}[proposition]{Lemma}

\theoremstyle{definition}
\newtheorem{definition}[proposition]{Definition}
\newtheorem{remark}[proposition]{Remark}

\title{The Boundary Harnack principle on optimal domains}

\author{Francesco Paolo Maiale, Giorgio Tortone, Bozhidar Velichkov}

\address {Francesco Paolo Maiale \newline \indent
	Scuola Normale Superiore\newline \indent
	Piazza dei Cavalieri 7, 56126 Pisa, Italy}
\email{francesco.maiale@sns.it}

\address {Giorgio Tortone \newline \indent
	Dipartimento di Matematica, Universit\`a di Pisa \newline \indent
	Largo Bruno Pontecorvo, 5, I--56127 Pisa, Italy}
\email{giorgio.tortone@dm.unipi.it}

\address {Bozhidar Velichkov \newline \indent
Dipartimento di Matematica, Universit\`a di Pisa \newline \indent
Largo Bruno Pontecorvo, 5, I--56127 Pisa, Italy}
\email{bozhidar.velichkov@unipi.it}

\thanks{{\bf Acknowledgments.}
GT and BV are supported by the European Research Council (ERC), under the European Union's Horizon 2020 research and innovation programme, through the project ERC VAREG - \it Variational approach to the regularity of the free boundaries \rm (grant agreement No. 853404).}


\begin{document}

\begin{abstract}
We give a short and self-contained proof of the Boundary Harnack Inequality for a class of domains satisfying some geometric conditions given in terms of a state function that behaves as the distance function to the boundary, is subharmonic inside the domain and satisfies some suitable estimates on the measure of its level sets. We also discuss the applications of this result to some shape optimization and free boundary problems.
\end{abstract}
\maketitle

\section{Introduction}\label{s:intro}
In this paper, we prove a Boundary Harnack Inequality  for domains satisfying some geometric conditions, which naturally arise in shape optimization and free boundary problems. One consequence of our analysis is that if a domain $\Omega\subset B_1$ admits a function, which is harmonic in $\Omega$, vanishes on $\partial\Omega$, behaves as the distance function to the boundary and the measure of its level sets decays linearly (see \cref{t:main} for the complete list of hypotheses), then the Boundary Harnack Inequality  holds on $\Omega$. This general principle is well-known in the free boundary community and was used for instance in \cite{ACS}, \cite{CSY} and \cite{mtv, mtv2}. In all these cases the strategy of the proof is to show that the optimal domain $\Omega$ is NTA and then to obtain the Boundary Harnack Inequality  by applying the well-known result of Jerison and K\"enig \cite{JK}. \smallskip

In this paper we give a direct proof of the Boudary Harnack inequality, without passing through the result for NTA domains (\cite{JK}). Our proof is completely self-contained and essentially uses only the mean value formula for harmonic functions and the classical Alt-Caffarelli-Friedman monotonicity formula for subharmonic functions. In \cref{s:harnack} we prove interior Harnack inequalities, which are the key part of the proof and encode the geometric properties of the domains. In \cref{sub:proof-of-T1}, we prove the Boundary Harnack Inequality  (\cref{t:main}); we follow step-by-step the strategy from the recent paper of De Silva and Savin \cite{DS} and the results from \cref{s:harnack}. In \cref{s:proof-of-T2}, for the sake of completeness, we show how to deduce the Boundary Harnack Principle (\cref{d:boundary-harnack}) from the Boundary Harnack Inequality  (\cref{t:main2}).\smallskip

The Boundary Harnack principle is a key tool in proving the $C^{1,\alpha}$ regularity of free boundaries arising in vectorial free boundary and shape optimization problems. We discuss some applications in \cref{sub:applications}.\smallskip

Throughout this paper $\Omega$ will be an open subset of the unit ball $B_1\subset\R^d$.
\begin{definition}[Boundary Harnack Principle]\label{d:boundary-harnack}
We say that the Boundary Harnack Principle holds in $\Omega$, if there is a constant $\alpha>0$ such that, for every
$$u:B_1\to\R\qquad\text{and}\qquad v:B_1\to\R$$
which are:
\begin{itemize}
\item continuous on $B_1$,
\item positive and harmonic in $\Omega\cap B_1$,
\item vanishing identically on $B_1\setminus\Omega$,
\end{itemize}	
the ratio $\,\ds\frac{u}{v}:\Omega\to\R\,$
can be extended to a $C^{0,\alpha}$-regular function on $B_{\rho}\cap\overline\Omega$, for some $\rho\in (0,1/2)$.
\end{definition}

The aim of this paper is to prove the following theorem.

\begin{theorem}\label{t:main}	
Let $\Omega\subset B_1$ be an open set with $0\in\partial\Omega$ and $\phi:B_1\to\R$ a continuous function such that:
\begin{enumerate}[\rm(a)]
\item\label{item:1-positivity} $\phi> 0$ on $\Omega$ and $\phi\equiv0$ on $B_1\setminus\Omega$;
\item\label{item:2-lipschitz} $\phi$ is $L$-Lipschitz continuous on $B_1$, where $L>0$ is a given constant;
\item\label{item:3-distance} $\phi$ behaves as the distance function to the set $B_1\setminus\Omega$; precisely, there is a constant $\kappa>0$ such that
$$\phi\ge \kappa \,\text{\rm dist}_{B_1\setminus\Omega}\quad\text{in}\quad B_{\sfrac12}\ ;$$
\item\label{item:4-subharmonicity} we have the inequality
$$\Delta \phi\ge 0\quad\text{in sense of distributions in }B_1;$$
\item\label{item:5-density} there is a constant $\mu>0$ such that for every $x_0\in\partial\Omega\cap B_{1}$, we have
$$|B_r(x_0)\setminus \Omega|\ge \mu |B_r(x_0)|\quad\text{for every}\quad r\in(0,1-|x_0|);$$
\item\label{item:6-levels} there is a constant $\Lambda>0$ such that for every $x_0\in\partial\Omega\cap B_{1}$ and every $r\in(0,1-|x_0|)$, we have
$$\big|\{0<\phi<rt\}\cap B_{r}(x_0)\big|\le \Lambda t|B_r|\quad\text{for every}\quad t>0.$$
\item\label{item:7-nondegeneracy} there is a constant $\eta>0$ such that for every $x_0\in\partial\Omega\cap B_{1}$ and every $r\in(0,1-|x_0|)$, we have
$$\sup_{B_r(x_0)}\phi\ge \eta r.$$
\end{enumerate}	
Then the Boundary Harnack Principle holds in $\Omega$ in the sense of \cref{d:boundary-harnack}.
\end{theorem}

In \cref{s:proof-of-T2}, we will deduce \cref{t:main} from the following theorem.

\begin{theorem}[Boundary Harnack Inequality ]\label{t:main2}	
Suppose that $\Omega\subset B_1$, $0\in\partial\Omega$, and $\phi:B_1\to\R$ satisfy the conditions \ref{item:1-positivity}, \ref{item:2-lipschitz}, \ref{item:3-distance}, \ref{item:4-subharmonicity}, \ref{item:5-density}, \ref{item:6-levels} and \ref{item:7-nondegeneracy} of \cref{t:main}. Then, there are constants $M>0$, $\delta\in(0,\eta]$ and $0<\rho<R\le 1$, depending on the dimension $d$ and the constants from \ref{item:2-lipschitz}, \ref{item:3-distance}, \ref{item:5-density}, \ref{item:6-levels} and \ref{item:7-nondegeneracy}, such that the following Boundary Harnack Inequality  holds. Suppose that
$$u,v:B_1\to\R$$
are nonnegative continuous functions satisfying
\begin{equation}\label{e:main2:hypothesis}
\begin{cases}
\Delta u=\Delta v=0\quad\text{in}\quad \Omega\cap B_1;\\
u=v=0\quad\text{on}\quad B_1\setminus \Omega;\\
u(P)=v(P)\quad\text{for some point}\quad P\in B_R\cap\{\phi>\delta R\}.
\end{cases}
\end{equation}
Then
$$\frac1M v\le u\le Mv\quad\text{in}\quad B_{\rho}\ .$$
\end{theorem}

In general, The Boundary Harnack Principle (\cref{d:boundary-harnack}) on a domain $\Omega$ is a consequence of the validity of the Boundary Harnack Inequality  at any scale and for any couple of nonnegative  functions $u,v$ satisfying \eqref{e:main2:hypothesis} on a rescaling of $\Omega$. This implication is well-known (see for instance \cite{JK}) and in \cref{s:proof-of-T2} we give a short proof of this fact in our context. In order to do so, we need that the Boundary Harnack Inequality  holds at any scale. This follows from the fact that the assumptions of \cref{t:main} and \cref{t:main2} are scale-invariant: \smallskip

\begin{remark}[Scale invariance]
Let $\Omega$ and $\phi$ be as in \cref{t:main}. Then, for every $x_0\in\partial\Omega\cap B_1$ and every $r\in(0,1-|x_0|)$, the rescalings $\Omega_{r,x_0}\subset B_1$ and $\phi_{r,x_0}:B_1\to\R$ defined as
$$\Omega_{r,x_0}:=\frac1r(-x_0+\Omega)\qquad\text{and}\qquad\phi_{r,x_0}(x):=\frac{\phi(x_0+rx)}{r}\,,$$
satisfy the properties \ref{item:1-positivity}, \ref{item:2-lipschitz}, \ref{item:3-distance}, \ref{item:4-subharmonicity}, \ref{item:5-density} and \ref{item:6-levels} of \cref{t:main} with the same constants. 					
\end{remark}

\begin{remark}[On the assumption \ref{item:7-nondegeneracy}]
We also notice that the assumption \ref{item:7-nondegeneracy} is only needed to assure that, for $\delta$ small enough, the set $B_R\cap \{\phi>\delta R\}$ from \eqref{e:main2:hypothesis} is non-empty. Indeed, if $\Omega$ is an open set in $B_1$ and $\phi:B_1\to\R$ is a function satisfying \ref{item:1-positivity} and \ref{item:7-nondegeneracy} of \cref{t:main}, then
	$$B_r(x_0)\cap\{\phi>r\delta\}\neq\emptyset\,\quad \text{for every}\quad x_0\in\partial\Omega\cap B_1,\quad r\in(0,1-|x_0|)\quad\text{and}\quad \delta\in(0,\eta).$$
\end{remark}

\begin{remark}
Several versions of the Boundary Harnack Inequality  (B.H.I.) appeared recently in the literature. See for instance \cite{LinLin}, where the authors established a B.H.I. on the class of nodal domains of solutions to uniformly elliptic equations in divergence form; \cite{AS} where B.H.I. was proved for solutions with right-hand side on sufficiently flat Lipschitz domains; we also refer to \cite{DS2} for a higher order Boundary Harnack Principle.
\end{remark}

\section{Harnack chains and interior Harnack inequalities}\label{s:harnack}
In this Section we prove the existence of Harnack chains, and consequently the validity of Harnack-type inequalities, by differentiating between those points that are close to the boundary and those that are away.
\subsection{Harnack chains and Harnack inequality close to the boundary}
In this subsection we show how to construct short Harnack chains starting from a point close to the boundary of a domain $\Omega$ satisfying the hypotheses of \cref{t:main}. This is done in the following simple lemma, which is essential for the proof of \cref{t:main2} (see \cref{s:proof-of-T2}, \cref{l:step1}).
\begin{lemma}[Short Harnack chains close to the boundary]\label{l:main}	
Suppose that $\Omega\subset B_1$, $0\in\partial\Omega$, and $\phi:B_1\to\R$ satisfy the conditions \ref{item:1-positivity}, \ref{item:2-lipschitz}, \ref{item:3-distance} and \ref{item:4-subharmonicity} of \cref{t:main}. Let
$$x_0\in \{\phi>0\}\cap B_{1}\qquad\text{be such that}\qquad 3\,\text{\rm dist}(x_0,\partial\Omega)<1-|x_0|,$$
let $r$ be the distance from $x_0$ to $\partial\Omega$ and $z_0$ be a projection of $x_0$ on $\partial\Omega$ (thus, $z_0\in\partial\Omega\cap B_1$ and $|x_0-z_0|=r$).

 Then, there is $y_0\in \partial B_r(x_0)$ such that the following holds:
\begin{enumerate}[\rm(i)]
\item $\phi(y_0)\ge (1+\sigma)\phi(x_0)$, where the constant $\sigma>0$, depends only on the dimension $d$, and the constants $L$ from \ref{item:2-lipschitz} and $\kappa$ from \ref{item:3-distance};
\item there is a constant $\Lambda >1$, depending only on the dimension $d$, the constants $L$ from \ref{item:2-lipschitz} and $\kappa$ from \ref{item:3-distance},  such that for every positive harmonic function $w:\Omega\to\R$,
$$\frac1{\Lambda} w(y_0)\le w(x_0)\le {\Lambda}w(y_0).$$
\end{enumerate}	
\end{lemma}
\begin{proof}
We fix a parameter $\eps>0$ that we will choose later on. First, we notice that by the condition \ref{item:4-subharmonicity}, we have that
\begin{align*}
\mean{\partial B_r(x_0)}\phi-\phi(x_0)&=\frac{1}{d\omega_d}\int_0^r s^{1-d}\Delta \phi(B_s(x_0))\,ds\ge 0.
\end{align*}
Let now the radius $\rho>0$ be such that
$$\mathcal H^{d-1}\big(\partial B_r(x_0)\cap B_\rho(z_0)\big)=\eps^{d-1}\mathcal H^{d-1}\big(\partial B_r(x_0)\big).$$
Now, for $\eps$ small enough $\rho$ is comparable to $\eps r$. In particular, by choosing $\eps$ small enough (depending only on the dimension), we have $\rho\le 2\eps r$, so the Lipschitz continuity of $\phi$ gives that
$$\phi(x)\le 2L\eps r\quad\text{for}\quad x\in \partial B_r(x_0)\cap B_\rho(z_0),$$
where $L$ is the Lipschitz constant from \ref{item:2-lipschitz}. Thus, setting
$$M:=\max\big\{\phi(x)\ :\ x\in \partial B_r(x_0)\big\},$$
we get that
\begin{align*}\phi(x_0)\le \mean{\partial B_r(x_0)}\phi &\le \frac{1}{r^{d-1}}\Big((2\eps r)^{d-1}2L\eps r+M\big(r^{d-1}-(2\eps r)^{d-1}\big)\Big)\\
&\le \frac{1}{r^{d-1}}\Big((2\eps r)^{d-1}\frac{2L\eps}{\kappa}\phi(x_0) +M\big(r^{d-1}-(2\eps r)^{d-1}\big)\Big)\\
&\le(2\eps)^{d-1}\frac{2L\eps}{\kappa}\phi(x_0) +M\big(1-(2\eps) ^{d-1}\big),
\end{align*}
which implies that
$$\Big(1-\eps ^{d-1}\frac{2L\eps}{\kappa}\Big)\,\phi(x_0)\le \big(1-\eps^{d-1}\big)M.$$
We now choose $\eps$ such that
$$\frac{2L\eps}\kappa\le \frac{1}{2^{d-1}}.$$
Thus, there is a point $y_0\in\partial B_r(x_0)$ such that
$$(1+\sigma)\phi(x_0)\le \phi(y_0),$$
where
$$1+\sigma:=\frac{1}{1-(2\eps)^{d-1}}\Big(1-\eps^{d-1}\Big).$$
In order to prove (ii), we notice that by the Lipschitz continuity of $\phi$, we have
$$\text{dist}_{B_1\setminus\Omega}(y_0)\ge \frac1{L}\phi(y_0)\ge \frac1{L}\phi(x_0)\ge \frac{\kappa}{L}r.$$
Thus,
$$B_r(x_0)\cap B_{\sfrac{r\kappa}{L}}(y_0)\subset \Omega,$$
and the claim (ii) follows by the classical Harnack inequality.
\end{proof}	

As a consequence, by iterating this result, we obtain the following Harnack-type inequality close the boundary.
\begin{lemma}[Interior Harnack inequality close to the boundary]\label{l:iteration}	
As in \cref{l:main}, we suppose that $\Omega\subset B_1$, $0\in\partial\Omega$, and $\phi:B_1\to\R$ satisfy the conditions \ref{item:1-positivity}, \ref{item:2-lipschitz}, \ref{item:3-distance} and \ref{item:4-subharmonicity} of \cref{t:main}. Then, there are constants $A>0$ and $\delta_0>0$, depending only on $d$ and the constants $L$ from \ref{item:2-lipschitz} and $\kappa$ from \ref{item:3-distance}, such that for every positive harmonic function $w:\Omega\to\R$, we have
$$\sup_{B_{\sfrac12}\cap \{\phi>\frac{\delta}{2}\}} w\le A\,\sup_{B_{1}\cap \{\phi>\delta\}} w\qquad\text{and}\qquad \inf_{B_{\sfrac12}\cap \{\phi>\frac{\delta}{2}\}} w\ge \frac1A\,\inf_{B_{1}\cap \{\phi>\delta\}} w.$$
for every $\delta\in(0,\delta_0]$.
\end{lemma}
\begin{proof}
Let $x_0\in B_{\sfrac12}\cap \{\phi>\frac\delta2\}$. If $\phi(x_0)>\delta$, then for any $A\ge 1$, we clearly have the inequalities
$$\frac{1}{A}\inf_{B_1\cap \{\phi>\delta\}} w\le w(x_0)\quad\mbox{and}\quad w(x_0)\le A\sup_{B_{1}\cap \{\phi>\delta\}} w,$$
Thus, we consider the case $x_0\in B_{\sfrac12}\cap \{\frac\delta2<\phi\le \delta\}$. Let $x_1$ be the point $y_0$ from \cref{l:main}. Then,
$$\phi(x_1)\ge (1+\sigma)\phi(x_0)\ge (1+\sigma)\frac{\delta}{2}.$$
Now, by construction $x_1\in B_r(x_0)$ and $r=\mathrm{dist}_{B_1\setminus \Omega}(x_0)$, so we get that $|x_1-x_0|=\mathrm{dist}_{B_1\setminus \Omega}(x_0)$.\\
Using this and \ref{item:3-distance}, we obtain that
$$|x_0-x_1|=\text{dist}_{B_1\setminus\Omega}(x_0)\le \frac1{\kappa}\phi(x_0)\le \frac{\delta}{\kappa}.$$
If $x_1\in \{\phi\le \delta\}$, we repeat the same procedure to obtain a point $x_2$. Iterating this argument, we obtain a sequence of points $x_n$ such that
$$x_n\in B_{r_n}\cap\Big\{\frac{\delta}{2}(1+\sigma)^n<\phi\le \delta\Big\}\qquad\text{with}\qquad r_n:=\frac12+n\frac{\delta}{\kappa}\,$$
and
$$\frac1{\Lambda^n}w(x_n)\le w(x_0)\le \Lambda^n w(x_n)$$
where $\Lambda>1$ is the Harnack constant from part (ii) of \cref{l:main}. Now, define $N$ to be the largest index for which $x_N\in B_{1}\cap \{\phi\le\delta\}$ and to which we can apply \cref{l:main} to obtain $x_{N+1}$. Thus, necessarily
$$\frac12(1+\sigma)^N\le 1,$$
which means that
$$N\le \frac1{\log_2(1+\sigma)}.$$
Thus, we have also that
$$r_{N+1}\le \frac12+(N+1)\frac{\delta}{\kappa}\le \frac12+\Big(\frac1{\log_2(1+\sigma)}+1\Big)\frac{\delta_0}{\kappa},$$
so by choosing $\delta_0$ small enough, we can suppose that $r_{N+1}\le 3/4$ and that we can still apply \cref{l:main} to $x_{N+1}$ to obtain $x_{N+2}$. Thus, the procedure stops because $$x_{N+1}\in\{\phi>\delta\}.$$
Hence
$$
\frac1{\Lambda^{N+1}} \min_{B_{1}\cap \{\phi>\delta\}} w\le \frac1{\Lambda^{N+1}}w(x_{N+1})\le w(x_0)
$$
and $$
w(x_0)\le \Lambda^{N+1}w(x_{N+1})\le \Lambda^{N+1} \max_{B_{1}\cap \{\phi>\delta\}} w.$$
The claim follows by taking $A:=\Lambda^{N+1}$ and $x_0$ as the point at which the maximum (resp. the minimum) $w$ is achieved in $B_1\cap \{\phi\ge \delta/2\}$.
\end{proof}	

\subsection{Harnack chains and Harnack inequality away from the boundary} The main result of this subsection is the following interior Harnack inequality away from the boundary, which we will use in Step 2 (\cref{sub:step2}) of the proof of \cref{t:main2}.
\begin{proposition}[Interior Harnack inequality away from the boundary]\label{l:interior-harnack}
	Suppose that $\Omega\subset B_1$, $0\in\partial\Omega$, and $\phi:B_1\to\R$ satisfy the conditions \ref{item:1-positivity}, \ref{item:2-lipschitz}, \ref{item:4-subharmonicity} and \ref{item:5-density} of \cref{t:main}. Then, for every $\delta>0$ there is $R_0$ for which the following holds.
	
	 For every $R\in(0,R_0]$, there is a constant $c_{\mathcal H}=c_{\mathcal H}(\delta, R)>0$ such that for every positive harmonic function
	$$w:\Omega\cap B_1\to\R\,,\qquad w\ge 0\quad\text{in}\quad \Omega\cap B_1\,,\qquad  \Delta w= 0\quad\text{in}\quad \Omega\cap B_1$$
	we have
	$$\inf_{\{\phi>\delta R\}\cap B_R} w\ge c_{\mathcal H}\,\sup_{\{\phi>\delta R\}\cap B_R} w.$$
\end{proposition}	
In order to prove \cref{l:interior-harnack} it is sufficient to show that there are constants $N>0$ and $r>0$ (depending also on $\delta$ and $R$) such that, for every pair of points $x_0,y_0\in \{\phi>\delta R\}\cap B_R$, there exists a curve $\gamma:[0,1]\to B_1$ such that
$$\gamma(0)=x_0\,;\quad \gamma(1)=y_0.$$
and a family of balls $\big\{B_{r}(x_j)\,:\,j=1,\dots,N\big\}$ such that:
\begin{itemize}
\item $x_j\in \gamma([0,1])$ for every $j=1,\dots,N$;
\item $B_{2r}(x_j)\subset\Omega$ for every $j=1,\dots,N$;
\item the family $\big\{B_{r}(x_j)\,:\,j=1,\dots,N\big\}$ is an open covering of $\gamma([0,1])$.
\end{itemize}	
The existence of such a family is an immediate consequence of the following lemma (and a covering theorem), in which we prove the existence of an Harnack chain by combining \ref{item:4-subharmonicity} with the monotonicity formula of Alt-Caffarelli-Friedman (see \cite{acf}).

\begin{lemma}[Harnack chains away from the boundary]\label{l:harnack-chains-away}
Suppose that $\Omega\subset B_1$, $0\in\partial\Omega$, and $\phi:B_1\to\R$ satisfy the conditions \ref{item:1-positivity}, \ref{item:2-lipschitz}, \ref{item:4-subharmonicity} and \ref{item:5-density} of \cref{t:main}. For every $\delta\in(0,2L)$ there is $\tau \in (0,1)$ such that the following holds. For every $R\in(0,1)$ and every couple of points $x_1,x_2\in B_{\tau R}\cap\{\phi>\delta R\}$, there is a curve connecting $x_1$ to $x_2$ in $B_R\cap \{\phi>\frac\delta2R\}.$
\end{lemma}
\begin{proof}
Fix $\tau \in (0,1)$ and suppose that $x_1$ and $x_2$ are two points in $B_{\tau R}\cap\{\phi>\delta R\}$ that lie in two different connected components, $\Omega_1$ and $\Omega_2$, of $B_R\cap \{\phi>\frac\delta2 R\}$ . Let $\phi_1$ and $\phi_2$ be the restrictions of the function $(\phi-\frac\delta2 R)_+$ respectively on $\Omega_1$ and $\Omega_2$. Then, $\phi_1$ and $\phi_2$ are both $L$-Lipschitz, $L$ being the constant from \ref{item:2-lipschitz}, and $\phi_{j}(x_j)\ge \frac{\delta}{2}R$ for $j=1,2$. Moreover, for every radius $r\in[\tau R,R]$, there is a point $x_r\in\partial B_r$ such that $\phi_1(x_r)=\phi_2(x_r)=0$. Define now the functions $\psi_j=(\phi_j-\frac{3\delta}{4}R)_+$ for $j=1,2$. Again $\psi_j$ are $L$-Lipschitz and harmonic where they are positive; we have that, $\psi_{j}(x_j)\ge \frac{\delta}{4}R$ and moreover
$$\psi_j\equiv 0\quad\text{on}\quad B_{\frac{\delta R}{4L}}(x_r)\quad\text{for every}\quad r\in[\tau R,R].$$
Now, if $\delta$ is small enough, this implies the density estimate
$$\alpha(r):=\frac{\mathcal H^{d-1}\Big(\{\psi_1=\psi_2=0\}\cap \partial B_r\Big)}{\mathcal H^{d-1}(\partial B_r)}\ge F\left(\frac{\delta^{d-1}}{(4L)^{d-1}}\right)\quad\text{for every}\quad r\in[\tau R,R],$$
$F:[0,+\infty)\to\R$ being a continuously differentiable increasing function depending only on the dimension and such that $F(0)=0$ and $F'(0)>0$. Now, for every $r\in[\tau R,R]$, let
$$
\Phi(r):=\frac{1}{r^4}\int_{B_r}\frac{|\nabla \psi_1|^2}{|x|^{d-2}}\,dx\int_{B_r}\frac{|\nabla \psi_2|^2}{|x|^{d-2}}\,dx.
$$
Now, by \cite{acf} (see also \cite[Lemma 4.3]{ACS}), we have that
$$\frac{d}{dr}\Big[\ln (\Phi(r))\Big]\ge \frac1rG(\alpha(r)),$$
where $G:[0,+\infty)\to\R$ is a positive increasing convex function with $G(0)=0$ and $G'(0)>0$. Combining the two estimates, we have that for $\delta$ small enough ($\delta\le\delta_0$ for some dimensional $\delta_0>0$),
$$\frac{d}{dr}\Big[\ln (\Phi(r))\Big]\ge C_d\frac1r\frac{\delta^{d-1}}{(4L)^{d-1}}.$$
since $\phi_1(0)=0=\phi_2(0)$ we get that
$$\mean{B_r}{|\nabla \psi_1|^2\,dx}\mean{B_r}{|\nabla \psi_2|^2\,dx}\le \Phi(r)\le \left(\frac{r}{R}\right)^{\alpha}\Phi(R),\quad\text{for }r\in [\tau R,R],$$
with $\alpha=C_d(\delta/(4L))^{d-1}$. Moreover, using again the density estimate \ref{item:5-density}, by the Poincar\'{e} inequality we deduce that
 $$\frac{1}{r^4}\mean{B_{r}}{\psi_1^2\,dx}\mean{B_{r}}{\psi_2^2\,dx} \le \left(\frac{r}{R}\right)^{\alpha}\Phi(R), \quad\mbox{for }r \in [\tau R,R].$$
We will next estimate the left-hand side from below. By the Lipschitz continuity \ref{item:2-lipschitz} of $\psi_1$ and $\psi_2$, we have that
 $$\psi_i\ge \frac{\delta}{4} R-|x-x_i|L\quad\text{in}\quad B_{\rho}(x_i).$$
Next, we choose $\rho=\frac{\delta R}{4L}$ and we notice that since $\frac{\delta}{L}\le 2$, we have that $\rho\le R/2$. Thus,
 $$\int_{B_{r}}\psi_i^2\,dx\ge \int_{B_{r}\cap B_\rho(x_i)}\psi_i^2\,dx\ge  \int_{B_{r}\cap B_\rho(x_i)}\left(\frac{\delta}{4} R-|x-x_i|L\right)^2\,dx\ge c_d\int_{B_\rho(x_i)}\left(\frac{\delta}{4} R-|x-x_i|L\right)^2\,dx\,,$$
 where $c_d$ is a dimensional constant. Now, a straightforward computation gives
 \begin{align*}
 \int_{B_\rho(x_i)}\left(\frac{\delta}{4} R-|x-x_i|L\right)^2\,dx&\ge \frac1{|B_\rho|}\left(\int_{B_\rho}\left(\frac{\delta}{4} R-|x|L\right)\,dx\right)^2\\
 &=\frac{(d\omega_d)^2}{\omega_d\rho^d}\left(\int_0^\rho s^{d-1}\left(\frac{\delta}{4} R-sL\right)\,ds \right)^2\\
 &=\frac{(d\omega_d)^2}{\omega_d\rho^d}\left(L\int_0^\rho s^{d-1}\left(\rho-s\right)\,ds \right)^2
 = \frac{\omega_d d^2}{(d+1)^2}L^2\rho^{d+2}.
 \end{align*}
 Thus,
 $$\frac{1}{r^2}\mean{B_{r}}{\psi_i^2\,dx}\ge C_d \frac{\delta^{d+2}}{L^d}\left(\frac{R}{r}\right)^{d+2}
 $$
 and so, by the Lipschitz continuity of $\psi_1,\psi_2$, we obtain the inequality
$$C_d\frac{\delta^{2d+4}}{L^{2d}}\le \left(\frac{r}{R}\right)^{d+2+\alpha} \Phi(R)\leq \left(\frac{r}{R}\right)^{d+2+\alpha} \frac{\omega_d^2 L^4}{4}  \quad \mbox{for }r \in [\tau R,R].$$
In particular, by taking $r=\tau R$ we deduce
$$
\tau^{d+2+\alpha}\geq C_d \left(\frac{\delta}{L}\right)^{2d+4},
$$
which is a contradiction if $\tau$ is small enough.
\end{proof}

\section{Boundary Harnack Inequality: proof of \cref{t:main2}}\label{sub:proof-of-T1}

In this section we prove \cref{t:main2}. We follow step-by-step the recent proof of De Silva and Savin of the Boundary Harnack Inequality  in Lipschitz and NTA domains \cite{DS}. The proof is divided in three main steps. In Step 1 (\cref{sub:step1}), the main result is \cref{l:step1} from which \cref{t:main2} follows by an iteration procedure; in our case \cref{l:step1} is an immediate consequence of \cref{l:iteration} from \cref{s:harnack}. In Step 2 (\cref{sub:step2}), we prove \cref{l:step2}, which allows to start the iteration procedure from Step 1. The proof of \cref{l:step2} is a consequence of \cref{l:growth2} from \cref{s:harnack} and on the Harnack-type estimate \cref{l:krylov-safonov}; for general operators \cref{l:krylov-safonov} is contained the proof of the Krylov-Safonov's Theorem \cite{KS} (see also \cite[Theorem 4.8]{CC} and \cite[Theorem 1.3]{DS}), while in our case it is a consequence of the mean-value formula. Finally, in Step 3 (\cref{sub:step3}), we simply combine the results from Step 1 and Step 2.

\subsection{Step 1}\label{sub:step1} The main result of this step is \cref{l:step1}; the proof is based on \cref{l:main} and the oscillation lemma from the De Giorgi's theorem.

\begin{lemma}\label{l:step1}
Suppose that $\Omega\subset B_1$, $0\in\partial\Omega$, and $\phi:B_1\to\R$ satisfy the conditions \ref{item:1-positivity},  \ref{item:2-lipschitz},  \ref{item:3-distance}, \ref{item:4-subharmonicity}, and  \ref{item:5-density} of \cref{t:main}. Then, there are constants  $\delta_1>0$ and $a>0$, depending on the dimension and the constants from  \ref{item:2-lipschitz},  \ref{item:3-distance} and  \ref{item:5-density}, for which that the following holds. 	
Suppose that $w:B_1\to\R$ is a continuous function satisfying
\begin{equation}
\begin{cases}
\begin{array}{rcl}
\Delta w=0&\quad\text{in}\quad &B_1\cap \{\phi>0\}\\
w=0&\quad\text{on}\quad &B_1\cap \{\phi=0\}\\
w\ge M&\quad\text{on}\quad &B_1\cap \{\phi> \delta\}\\
w\ge -1&\quad\text{on}\quad &B_1\cap \{0<\phi\le\delta\},
\end{array}
\end{cases}
\end{equation}
for some $\delta\in(0,\delta_1]$ and some $M>0$.
Then, in $B_{\sfrac{1}2}$,
\begin{equation}
\begin{cases}
\begin{array}{rcl}
\Delta w=0&\quad\text{in}\quad &B_{\sfrac{1}2}\cap \{\phi>0\}\\
w=0&\quad\text{on}\quad &B_{\sfrac{1}2}\cap \{\phi=0\}\\
w\ge aM&\quad\text{on}\quad &B_{\sfrac{1}2}\cap \{\phi>  \frac{\delta}2\}\\
w\ge -a&\quad\text{on}\quad &B_{\sfrac{1}2}\cap \{0<\phi\le \frac{\delta}{2}\}.
\end{array}
\end{cases}
\end{equation}
Therefore, we get that
$$
\sup_{B_{1/2}\cap \Omega }w^- \leq a
\quad\mbox{and}\quad\inf_{B_{1/2}\cap \{\phi>\delta/2\}}w^+ \geq a M 
$$
where $w^+, w^-$ are respectively the positive and negative part of $w$.
\end{lemma}	
\begin{proof}
We consider the function $w+1$, which is positive and harmonic on $B_1\cap\Omega$.\\	
Taking $\delta_1$ to be smaller than the constant $\delta_0$ from \cref{l:iteration}, we have
$$\min_{B_{\sfrac{1}{2}}\cap\{\phi>\sfrac{\delta}{2}\}}(w+1)\ge \frac1A\min_{B_1\cap\{\phi>\delta\}}(w+1)\ge \frac1A(M+1),$$
and so, if we choose
$$a\le\frac1{2A}\qquad\text{and}\qquad M=2A,$$
we get
$$\min_{B_{\sfrac{1}{2}}\cap\{\phi>\sfrac{\delta}{2}\}}w\ge \frac1A(M+1)-1\ge 1\ge aM.$$
In order to prove the bound from below on $B_{\sfrac{1}2}\cap \{0<\phi\le \frac{\delta}{2}\}$, we use
the density bound from \ref{item:5-density} and the classical De Giorgi's oscillation lemma (see \cite[Theorem 1.2]{DS} and \cref{oscillation}). In fact, if we fix a point
$$x_0\in B_{\sfrac12}\cap\partial\{\phi>0\}$$
and if we apply \cref{oscillation} to the negative part of $w$, then we get that
$$\sup_{B_{2^{-n}}(x_0)}w_-\le (1-c)^{n-1}\sup_{B_{\sfrac12}(x_0)}w_-\le (1-c)^{n-1}\sup_{B_{1}}w_-\le (1-c)^{n-1},$$
where $c\in(0,1)$ is the dimensional constant from \cref{oscillation} below. Now, choosing $n$ to be such that
$$(1-c)^{n-1}\le \frac1{2A},$$
we get
\begin{equation}\label{e:lower-est-new}
\sup_{B_{2^{-n}}(x_0)}w_-\le \frac1{2A}.
\end{equation}
We now choose the second bound on $\delta_1$ to be
$$\delta_1\le 2\kappa\, 2^{-n}.$$
Thus, by the bound from below \ref{item:3-distance}, we have that the set  $B_{\sfrac12}\cap\partial\{0<\phi<\sfrac{\delta_1}2\}$ is contained in the union of all balls $B_{2^{-n}}(x_0)$ with centers $x_0\in \partial\Omega\cap B_{\sfrac12}$. Thus,
choosing $a$ to be precisely $\frac1{2A}$ and using \eqref{e:lower-est-new}, we get that
$$w_-\le a\qquad\text{on}\qquad B_{\sfrac12}\cap\partial\{0<\phi<\sfrac{\delta}2\},$$
for any $\delta\le \delta_1$, which concludes the proof.
\end{proof}	

\begin{remark}[De Giorgi's oscillation lemma for the Laplacian]\label{oscillation}
Suppose that $w:B_1\to\R$ is a subharmonic function bounded between $0$ and $1$, and such that $|\{w=0\}\cap B_{\sfrac14}|\ge\mu|B_{\sfrac14}|$ for some constant $\mu>0$. Then,
\begin{equation}\label{e:oscillation-subharmonic}
w\le 1-c\quad\text{on}\quad B_{\sfrac12},
\end{equation}
where $c>0$ depends only on $\mu$ and the dimension $d$. Indeed, by the mean value formula, for every $x_0\in B_{\sfrac14}$
$$w(x_0)\le \frac1{|B_{\sfrac12}|}\int_{B_{\sfrac12}(x_0)}w(x)\,dx\le \frac1{|B_{\sfrac12}|}\Big(|B_{\sfrac12}|-\mu|B_{\sfrac14}|\Big)=1-\frac{\mu}{2^d}.$$
Now let $y_0\in B_{\sfrac12}$. Since $B_{\sfrac12}(y_0)\cap B_{\sfrac14}$ contains at least a ball of radius $\sfrac18$, by the previous estimate in $B_{1/4}$ we get that
$$w(y_0)\le \frac1{|B_{\sfrac12}|}\int_{B_{\sfrac12}(y_0)}w(x)\,dx\le \frac1{|B_{\sfrac12}|}\Big(|B_{\sfrac12}|-\frac{\mu}{2^d}|B_{\sfrac18}|\Big)=1-\frac{\mu}{8^d},$$
which is precisely \eqref{e:oscillation-subharmonic} with $c={8^{-d}}\mu$.
\end{remark}
	
\subsection{Step 2}\label{sub:step2}
In this section we prove a bound which allows to start the iterative procedure based on \cref{l:step1} from Step 1. This is the only point of the proof in which we use the hypothesis \ref{item:6-levels} of \cref{t:main2}.

\begin{proposition}\label{l:step2}
	Suppose that $\Omega\subset B_1$, $0\in\partial\Omega$, and that $\phi:B_1\to\R$ satisfies the conditions \ref{item:1-positivity}, \ref{item:2-lipschitz}, \ref{item:3-distance}, \ref{item:4-subharmonicity} and \ref{item:6-levels}  of \cref{t:main}. Then, there are constants $C>0$ and $\delta_2>0$ depending on  $d$ and the constants from \ref{item:2-lipschitz}, \ref{item:3-distance}, \ref{item:6-levels}, for which the following holds. If $w:B_1\to\R$ is a nonnegative  continuous function satisfying
	\begin{equation}
	\begin{cases}
	\begin{array}{rcl}
	\Delta w=0&\quad\text{in}\quad &B_1\cap \{\phi>0\}\\
	w=0&\quad\text{on}\quad &B_1\cap \{\phi=0\}\\
	w\le 1&\quad\text{on}\quad &B_1\cap \{\phi\ge \delta_2\}.
	\end{array}
	\end{cases}
	\end{equation}
	Then,
	$$w\le C\quad\text{in}\quad B_{\sfrac{1}4}.$$
\end{proposition}	

We first prove the following lemma which is a consequence of the Harnack inequality close to the boundary (\cref{l:iteration}). We notice that the constants $\delta_2$ from \cref{l:growth2} and \cref{l:step2} are the same.
\begin{lemma}[A pointwise estimate up to the boundary]\label{l:growth2}	
	Suppose that $\Omega\subset B_1$, $0\in\partial\Omega$, and $\phi:B_1\to\R$ satisfy the conditions \ref{item:1-positivity}, \ref{item:2-lipschitz}, \ref{item:3-distance} and \ref{item:4-subharmonicity} of \cref{t:main}. There are constants $\delta_2>0$, $C$ and $p$ depending on  $d$, $L$ and $\kappa$ from \ref{item:2-lipschitz} and \ref{item:3-distance}, for which the following holds. For every $\delta\in(0,\delta_2]$ and every positive harmonic function $w:\Omega\to\R$, satisfying
	$$w\le 1\quad\text{on}\quad B_1\cap\{\phi>\delta\},$$
	we have
	$$w\le C\phi^{-p}\quad\text{on}\quad B_{\sfrac12}\cap\Omega.$$
\end{lemma}
\begin{proof}
	It is sufficient to prove the claim for $\delta=\delta_2$. Let $x_0\in B_{\sfrac12}\cap\Omega$ and $\ell\ge 1$ be a fixed constant that we will choose later. If $\phi(x_0)\ge\ell\delta_2$, then it is enough to choose $C\ge L^p$. In fact, by the Lipschitz bound \ref{item:2-lipschitz} and the fact that $x_0\in\{\phi>\delta_2\}$, we have
	$$w(x_0)\le 1\le {L^p}\Big(\ds\max_{B_{\sfrac12}}\phi\Big)^{-p}\le {L^p}{\phi(x_0)^{-p}}.$$
	Therefore, suppose that $\phi(x_0)\le \ell\delta_2$. Let $z_0$ be the projection of $x_0$ on $\partial\Omega \cap B_1$. By \ref{item:3-distance}, we have that
$$
r:=|x_0-z_0|\le \frac1{\kappa}\phi(x_0)\le \frac{\ell\delta_2}{\kappa}\,.
$$
	Thus, if $\ell\delta_2$ is small enough, such that $\ell\delta_2\le \sfrac\kappa8$, we have that $r\le \sfrac18$ and, in particular, $B_{2r}(z_0)\subset B_1$. Moreover, by the bound from below \ref{item:3-distance} we have that $$\frac{\kappa}2 2r=\kappa|x_0-z_0|\le \phi(x_0)\,$$
	and so, since
	$\ell\delta_2\le \frac\kappa2$, we get that
	$$x_0\in B_{2r}(z_0)\cap\big\{\phi>\kappa r\big\}\subset  B_{2r}(z_0)\cap\big\{\phi>2\ell\delta_2 r\big\}.$$
	Let now $\ell\delta_2\le \delta_0$, $\delta_0$ being the threshold from \cref{l:iteration}, and let $n\ge 1$ be such that
	\begin{equation}\label{e:choice-of-n}
	2^nr\le \frac14< 2^{n+1}r.
	\end{equation}
	Then, $B_{2^n r}(z_0)\subset B_1$ and we can iterate the estimate from \cref{l:iteration} obtaining
	$$w(x_0)\le \max_{B_{2r}(z_0)\cap\{\phi>2r\ell\delta_2 \}}w\le A^{n-1}\max_{B_{2^nr}(z_0)\cap\{\phi>2^n\ell\delta_2 r\}} w\le A^{n-1}\max_{B_{1}\cap\{\phi>\frac{\ell\delta_2}8\}} w.$$
	Thus, let us choose $\ell=8$.\\
	Next, since
	$$2\leq \kappa \phi(x_0)^{-1},$$	
	by choosing $p>0$ such that $A^{n-1}=2^p>1$, we get that
	$$w(x_0)\le A^{n-1}\max_{B_{1}\cap\{\phi>\delta_2\}} w \le 2^p\max_{B_{1}\cap\{\phi>\delta_2\}} w\le {\kappa^p}\phi(x_0)^{-p},$$
	which gives the claim. We notice that it is enough to choose $\delta_2$ and $C$ as
	\begin{align*}
	\delta_2\le\min\Big\{\frac{\kappa}{64}, \frac{\delta_0}8\Big\}&\qquad\text{and}\qquad C=\max\{\kappa^p,L^p\}.\qedhere
	\end{align*}
\end{proof}	


In the proof of \cref{l:step2} we will need the following Krylov-Safonov-type estimate, which was also used in \cite{DS} (see \cite[Theorem 1.3]{DS}). In our specific case, there is a simple proof based only on the mean-value fromula for harmonic functions, which still uses the idea from the conclusion of the Krylov-Safonov's Theorem.

\begin{lemma}[A Krylov-Safonov-type estimate]\label{l:krylov-safonov}
Suppose that $\Omega$ is an open set in $B_1$ and that the continuous\footnote{We notice that this assumption is not restricitive as below we will also assume $w$ is harmonic in $\Omega$ and that $\Omega$ saisfies an exterior density bound.} function $w:B_1\to\R$ is such that:
\begin{itemize}
\item $w$ is nonnegative on $B_1$ and vanishes identically on $B_1\setminus\Omega$;
\item $w$ is harmonic in $\Omega$ and subharmonic on $B_1$;
\item $\Omega$ satisfies the exterior denisity bound \ref{item:5-density} in $B_1$;
\item there is $\eps>0$ such that $\ds\int_{B_1}w^\eps\,dx\le 1$.	
\end{itemize}	
Then, there is a constant $M>0$ depending on the dimension, the density bound $\mu$ from \ref{item:5-density} and on $\eps$, such that
$$w\le M\quad\text{in}\quad B_{\sfrac12}.$$
\end{lemma}	
\begin{proof}
Let $x_0\in B_{\sfrac12}\cap\Omega$, $R:=\text{\rm dist}(x_0,\partial\Omega)$ and $M:=w(x_0)>0$. We also fix $\delta:=\sfrac\eps{2d}$. We consider two cases. 	

\noindent{\bf Case 1.} Assume that $2R\ge M^{-\delta}$. Notice that, in $B_R(x_0)$ the function $w$ is harmonic (and positive). Thus, by the Harnack inequality in $B_R(x_0)$, there is a dimensional constant $c_{\mathcal H}>0$ such that
$$w\ge c_{\mathcal H} M\quad\text{in}\quad B_{R/2}(x_0).$$
But then,
$$1\ge \int_{B_1}w^\eps\,dx\ge \int_{B_{R/2}(x_0)}w^\eps\,dx\ge |B_{R/2}|(c_{\mathcal H}M)^\eps\ge \frac{\omega_d\,c_{\mathcal H}^\eps}{4^d}M^{-d\delta+\eps}= \frac{\omega_d\,c_{\mathcal H}^\eps}{4^d}M^{\sfrac{\eps}2},$$
which means that in this case there is a constant $C_{d,\eps}$ depending only on $d$ and $\eps$ such that $M\le C_{d,\eps}$.\smallskip

\noindent{\bf Case 2.} Suppose now that $2R\le M^{-\delta}$ and $M>C_{d,\eps}$, $C_{d,\eps}$ being the constant from the previous case.\\ Let $z_0$ be the projection of $x_0$ on $\partial\Omega\cap B_1$. Then, the ball $B_{M^{-\delta}}(x_0)$ contains $B_{M^{-\delta}/2}(z_0)$ and is contained in $B_1$. In particular, since $\Delta w\ge 0$ in $B_1$,
$$M=w(x_0)\le \frac1{|B_{M^{-\delta}}|}\int_{B_{M^{-\delta}}(x_0)}w(x)\,dx\le \frac{|B_{M^{-\delta}}(x_0)\cap\Omega|}{|B_{M^{-\delta}}|}\|w\|_{L^\infty(B_{M^{-\delta}}(x_0))},$$
which by the density estimate in the ball $B_{M^{-\delta}/2}(z_0)$ gives
$$M\le (1-2^{-d}\mu)\|w\|_{L^\infty(B_{M^{-\delta}}(x_0))} \le \frac1{1+2^{-d}\mu}\|w\|_{L^\infty(B_{M^{-\delta}}(x_0))},$$
which means that there exists a point $x_1\in B_{M^{-\delta}}(x_0)$ such that
$$w(x_1)\ge (1+2^{-d}\mu)M.$$
Iterating the same procedure, we obtain a sequence of points $x_n\in\Omega\cap B_1$ such that
$$w(x_{n+1})\ge (1+2^{-d}\mu)w(x_{n})\ge M(1+2^{-d}\mu)^n\qquad\text{and}\qquad |x_{n+1}-x_{n}|\le \frac{1}{M^\delta(1+2^{-d}\mu)^{n\delta}}.$$
Now, if we choose $M$ large enough, then
$$\sum_{n=0}^{+\infty}\frac{1}{M^\delta(1+2^{-d}\mu)^{n\delta}}\le \frac14,$$
so $x_n$ is defined for every $n\ge 1$ (it never leaves $\Omega\cap B_{\sfrac34}$). But this is impossible since $w(x_n)\to\infty$.
\end{proof}

\begin{proof}[\bf Proof of \cref{l:step2}]
	We first show that there are $\alpha>0$ and $C>0$ such that
	$$\int_{B_{\sfrac12}}w^\alpha\,dx\le C.$$
	Indeed, by \cref{l:growth2} and \ref{item:6-levels} of \cref{t:main}, we have that
	\begin{align*}
	\int_{B_{\sfrac12}}(C\phi^{-p})^\alpha\,dx&=C^\alpha p\alpha\int_0^{+\infty}t^{\alpha p-1}|\{\phi^{-1}>t\}\cap B_{\sfrac12}|\,dt\\
	&\le C^\alpha p\alpha\left(|B_{\sfrac12}|\int_0^1t^{\alpha p-1}\,dt+\int_1^{+\infty}t^{\alpha p-1}|\{\phi^{-1}>t\}\cap B_{\sfrac12}|\,dt\right)\\
	&\le C^\alpha p\alpha\left(\frac1{\alpha p}|B_{\sfrac12}|+\int_1^{+\infty}t^{\alpha p-1}|\{\phi<1/t\}\cap B_{\sfrac12}|\,dt\right)\\
	&\le C^\alpha p\alpha\left(\frac1{\alpha p}|B_{\sfrac12}|+|B_{\sfrac12}|\Lambda\int_1^{+\infty}t^{\alpha p-2}\,dt\right)=C^\alpha|B_{\sfrac12}|\left(1+\frac{\Lambda \alpha p}{1-\alpha p}\right),
	\end{align*}
	so it is sufficient to choose $\alpha=\frac1{2p}$. Now, the conclusion follows from \cref{l:krylov-safonov}.
\end{proof}	

\subsection{Step 3}\label{sub:step3}

We first show that we can choose the constants $M$ from \cref{t:main2} and a level $\delta$ in such a way that we can start the iterative procedure from \cref{l:step1}
\begin{lemma}\label{l:step3}	
	Suppose that $\Omega\subset B_1$, $0\in\partial\Omega$, and $\phi:B_1\to\R$ satisfy the conditions \ref{item:1-positivity}, \ref{item:2-lipschitz}, \ref{item:3-distance}, \ref{item:4-subharmonicity}, \ref{item:5-density} and \ref{item:6-levels} of \cref{t:main}. Let $R\in(0,R_0]$  where $R_0$ is the radius from \cref{l:interior-harnack}. Then, there are constants $C_\ast>0$ and $\delta\le \min\{\eta,\delta_1,\delta_2\}$\footnote{$\delta_1$ and $\delta_2$ are the constants from \cref{l:step1} and \cref{l:step2}, while $\eta$ is the constant from \ref{item:7-nondegeneracy} of \cref{t:main}.}, depending on the dimension $d$, the radius $R$, and the constants from \ref{item:2-lipschitz}, \ref{item:3-distance}, \ref{item:5-density} and \ref{item:6-levels}, such that for every couple
	$$u,v:B_1\to\R$$
	of nonnegative continuous functions satisfying
	$$\begin{cases}
	\Delta u=\Delta v=0\quad\text{in}\quad \Omega\cap B_1;\\
	u=v=0\quad\text{on}\quad B_1\setminus \Omega;\\
	u(P)=v(P)\quad\text{for some point}\quad P\in B_R\cap\{\phi>\delta R\},
	\end{cases}$$
	we have that
$$
C_\ast u-v \quad\mbox{and}\quad  C_\ast v - u
$$
fulfill the assumptions of \cref{l:step1}.
\end{lemma}

\begin{proof}[\bf Proof of \cref{l:step3}]
Indeed, by \cref{l:interior-harnack}, there is a constant $C$ (depending also on $R$) such that 	
$$\frac1{C}\le u,v\le C\quad\text{on}\quad B_R\cap\{\phi>\delta R\}.$$
Thus, by \cref{l:step2}, there is a constant $\Lambda>0$ such that
$$v\le \Lambda\quad\text{in}\quad B_{\sfrac{R}{4}},$$
and a constant $\lambda>0$ such that
$$u\le \lambda\quad\text{in}\quad B_{\sfrac{R}{4}}\cap\{\phi>\frac{\delta}4 R\}.$$
Thus, for some $C_1>0$ large enough, the function $C_1 u-v$ satisfies the conditions of \cref{l:step1}.\\Relabeling the previous inequalities, we easily deduce the existence of $C_2>0$ large enough, such that $C_2 v - u$ satisfies the assumption of \cref{l:step1} too. Finally, the result follows by taking $C_\ast=\max\{C_1,C_2\}$.
\end{proof}		
	
\begin{proof}[\bf Proof of \cref{t:main2}] We first notice that, by choosing $R$ and $\delta$ small enough, we can apply \cref{l:step3} in a neighborhood of the origin.  Precisely, there are $R>\rho>0$ and $\delta$, such that
$$
\sup_{B_{R/4}(x_0)\cap \Omega}(C_\ast u-v)^- \leq a
\quad\mbox{and}\quad\inf_{B_{R/4}(x_0)\cap \{\phi>\frac{R}{4}\delta\}}(C_\ast u-v)^+ > 0,
$$
for every $x_0\in B_\rho$. Iterating \cref{l:step1} (up to a dilatation and rescaling), we get that for every $n\ge 0$
	 $$C_\ast u-v\ge 0\quad\text{in}\quad B_{r_n}(x_0)\cap \{\phi>r_n\delta\},$$
	 where $r_n:=R2^{-2-n}$. Now, it is sufficient to notice that for $\rho$ small enough the family of sets
	 $$\Big\{ B_{r_n}(x_0)\cap \{\phi>r_n\delta\}\ :\ x_0\in\partial\Omega\cap B_{\rho},\ n\ge 0\Big\},$$
	 is a covering of $B_\rho$. By repeating the same argument with $C_\ast v - u$, we get the claimed result.
\end{proof}

%

\section{H\"older continuity up to the boundary. Proof of \cref{t:main}}\label{s:proof-of-T2}

In this section, we show how the Boundary Harnack Inequality (\cref{t:main2}) implies that the ratio of two harmonic functions vanishing simultaneously on $\partial\Omega$ is H\"older continuous up to the boundary (\cref{t:main}).  Our main theorem is a consequence of the following proposition, which is well-known (see for instance \cite[Corollary 1.3.8]{K}); we give here the detailed proof for the sake of completeness.

\begin{proposition}\label{p:main}
	Let $\Omega\subset B_1$ be an open set with the following property.
	\begin{equation}\label{e:BH-half-space-property}
	\begin{cases}\begin{array}{ll}
	\quad \text{There is a constant $M>0$ such that for every $x_0\in\partial\Omega\cap B_{\sfrac12}$,}\\
	\text{every $r\in(0,1-|x_0|)$, there is a point $P_r(x_0)\in B_r(x_0)\cap \Omega$ for which the following holds.}\\	
	\text{For every pair of continuous non-negative functions}\smallskip\\
	\qquad\qquad\qquad\qquad \qquad u,v:B_r(x_0)\to\R\\
	\text{satisfying}\\
	\qquad\qquad\qquad\qquad
	\Delta u=\Delta v=0\quad\text{in}\quad B_r(x_0)\cap \Omega\,,\\
	\qquad\qquad\qquad\qquad u=v=0\quad\text{on}\quad B_r(x_0)\setminus \Omega\,,\\
	\qquad\qquad\qquad\qquad u(P_r(x_0))=v(P_r(x_0))\,,\\
	\text{we have that}\\
	\ds\qquad\qquad\qquad\frac1M\le \frac{u(x)}{v(x)}\le M\quad\text{for every}\quad x\in B_{\sfrac{r}2}(x_0)\cap \Omega.
	\end{array}	
	\end{cases}
	\end{equation}	
	Then, there are constants $\alpha>0$ and $C>0$, depending on $c$, $M$ and the dimension, such that for every pair of continuous non-negative functions $$u,v:B_1\to\R$$ satisfying
	\begin{equation}
	\label{e:BH-half-space-hypo}
	\begin{cases}
	\Delta u=\Delta v=0\quad\text{in}\quad B_1\cap \Omega\\
	u=v=0\quad\text{on}\quad B_1\setminus\Omega\\
	u(P_1(0))=v(P_1(0))>0\,,
	\end{cases}
	\end{equation}
	the following H\"older estimate holds
	\begin{equation}\label{e:BH-half-space-conclusion}
	\left|\frac{u(x)}{v(x)}-\frac{u(y)}{v(y)}\right|\le C|x-y|^\alpha\qquad\text{for every}\qquad x,y\in B_{\sfrac14}\cap\Omega\,.
	\end{equation}
\end{proposition}

In order to prove the proposition it is sufficient to estimate the oscillation of $\frac{u}{v}$ from one scale to another. The main lemma is the following.
\begin{lemma}\label{l:mainfin}
	Let $\Omega\subset B_1$ be an open set with the property \eqref{e:BH-half-space-property}. Then, for every $x_0\in\partial\Omega\cap B_{\sfrac12}$, every $r\le \sfrac12$, and every pair of continuous and non-negative functions $u,v:B_r(x_0)\to\R$ satisfying
$$\Delta u=\Delta v=0\quad\text{in}\quad B_r(x_0)\cap \Omega\ ,\qquad u=v=0\quad\text{on}\quad B_r(x_0)\setminus \Omega,$$
we have that
$$\osc_{\Omega\cap B_{\sfrac{r}2}(x_0)}\frac{u}{v}\le \Big(1-\frac1{2M}\Big)\osc_{\Omega\cap B_r(x_0)}\frac{u}{v},$$
where $M$ is the constant from \eqref{e:BH-half-space-property}.
\end{lemma}
\begin{proof}[Proof of \cref{l:mainfin}]
For simplicity, we set
$$P_r:=P_r(x_0)\ ,\qquad M_r:=\sup_{\Omega\cap B_r(x_0)}\frac{u}{v}\qquad\text{and}\qquad m_r:=\inf_{\Omega\cap B_r(x_0)}\frac{u}{v}\,.$$
Suppose that $\ds \frac{u(P_r)}{v(P_r)}\ge \frac{M_r+m_r}{2}$. Then, the functions
$u-m_rv$ and $v$
are harmonic and non-negative in $B_r(x_0)\cap\Omega$ and satisfy
$$u(P_r)-m_r v(P_r)\ge \frac{M_r-m_r}2v(P_r).$$
Thus, by the hypothesis \eqref{e:BH-half-space-property}, we have
$$u-m_r v\ge \frac1M\frac{M_r-m_r}2v\qquad\text{in}\qquad B_{r/2}(x_0),$$
where $M$ is the constant from \eqref{e:BH-half-space-property}.
Thus,
$$\inf_{\Omega\cap B_{r/2}(x_0)}\frac{u}{v}\ge m_r+\frac1M\frac{M_r-m_r}2,$$
and so
$$\osc_{\Omega\cap B_{r/2}(x_0)}\frac{u}{v}\le M_r-\left( m_r+\frac1M\frac{M_r-m_r}2\right)=(M_r-m_r)\left(1-\frac1{2M}\right).\smallskip$$
Analogously, if $\ds \frac{u(P_r)}{v(P_r)}\le \frac{M_r+m_r}{2}$, then
$$M_r v-u\ge \frac1M\frac{M_r-m_r}2v\qquad\text{in}\qquad B_{r/2}(x_0),$$
which implies that
$$\sup_{\Omega\cap B_{r/2}(x_0)}\frac{u}{v}\le M_r-\frac1M\frac{M_r-m_r}2,$$
and
$$\osc_{\Omega\cap B_{r/2}(x_0)}\frac{u}{v}\le \left(M_r-\frac1M\frac{M_r-m_r}2\right)-m_r=(M_r-m_r)\left(1-\frac1{2M}\right).$$
which concludes the proof of \cref{l:mainfin}.
\end{proof}	

\begin{proof}[Proof of \cref{p:main}] We will prove the following claim.
\begin{equation}\label{e:BH-half-space-claim2}
\begin{cases}
\begin{array}{ll}
\text{There is a constant $c\in(0,1)$ such that for every $x_0\in\overline\Omega\cap B_{\sfrac12}$, every $r\le \sfrac12$,}\\
\text{and every pair of continuous and non-negative functions}\smallskip\\
\qquad\qquad\qquad\qquad\qquad\qquad \qquad u,v:B_r(x_0)\to\R\\
\text{satisfying}\\
\qquad\qquad\qquad\qquad
\Delta u=\Delta v=0\quad\text{in}\quad B_r\cap \Omega\ ,\qquad u=v=0\quad\text{on}\quad B_r\setminus \Omega,\\
\text{we have that}\\
\ds\qquad\qquad\qquad\qquad\qquad\qquad\osc_{\Omega\cap B_{\sfrac{r}{16}}(x_0)}\frac{u}{v}\le (1-c)\osc_{\Omega\cap B_r(x_0)}\frac{u}{v}.
\end{array}	
\end{cases}
\end{equation}	
In order to prove \eqref{e:BH-half-space-claim2}, we consider two cases. \smallskip

Suppose that there is a point $y_0\in \partial\Omega\cap B_{r/8}(x_0)$. Then, we have that $B_{r/2}(y_0)\subset B_r(x_0),$
and by \cref{l:mainfin}, we have
$$\osc_{B_{r/4}(y_0)\cap\Omega}\frac{u}{v}\le \Big(1-\frac1{2M}\Big)\osc_{B_{r/2}(y_0)\cap\Omega}\frac{u}{v}\le \Big(1-\frac1{2M}\Big)\osc_{B_{r}(x_0)\cap\Omega}\frac{u}{v}.$$
Now, since $B_{r/8}(x_0)\subset B_{r/4}(y_0)$, we get that
$$\osc_{B_{r/8}(x_0)\cap\Omega}\frac{u}{v}\le (1-c) \osc_{B_{r}(x_0)\cap\Omega}\frac{u}{v}\qquad\text{with}\qquad c=\frac1{2M}.\smallskip$$
Conversely, suppose that $B_{r/8}(x_0)\subset \Omega$. Then, by the classical (interior) Harnack inequality, we have
$$\osc_{B_{r/16}(x_0)\cap\Omega}\frac{u}{v}\le  (1-c_{\mathcal H})\osc_{B_{r/8}(x_0)\cap\Omega}\frac{u}{v}\le (1-c_{\mathcal H})\osc_{B_{r}(x_0)\cap\Omega}\frac{u}{v},$$
where $c_{\mathcal H}\in (0,1)$ is a dimensional constant. This concludes the proof of \eqref{e:BH-half-space-claim2}.
The H\"older estimate  \eqref{e:BH-half-space-conclusion} now follows by a standard argument.	
\end{proof}

\section{Applications}\label{sub:applications} In this section we briefly discuss two examples of domains satisfying the conditions of \cref{t:main}.

\subsection{The vectorial free boundary problem}\label{subsub:vectorial} Let $B_1\subset\R^d$. For every vector-valued function $U:B_1\to\R^k$ we define the functional
$$\mathcal F(U):=\int_{B_1}|\nabla U|^2\,dx+\big|\{|U|>0\}\big|.$$
We say that a function $U:B_1\to\R^k$ is a (variational) solution of the vectorial problem if it minimizes $\mathcal F$ among all $\R^k$-valued functions with prescribed values on $\partial B_1$. We say that $U=(u_1,\dots,u_k)$ is {\it non-degenerate}, if there is a component, say $u_1$, which is strictly positive in $\{|U>0|\}\cap B_1$.
If this is not the case, we say that $U$ is {\it degenerate}. The non-degenerate case was first studied in \cite{CSY, krli, mtv}, while the regularity of the flat free boundaries in the degenerate case was first obtained in \cite{krli2}; see also \cite{desilva-tortone} for a different approach and \cite{spve} for an analysis of the singular part of the free boundaries in dimension two.  \smallskip

We notice that the proofs in \cite{CSY, mtv,mtv2}, of the $C^{1,\alpha}$ regularity of the flat free boundaries, are all based on the Boundary Harnack principle, which allows to transform the free boundary condition
$$\sum_{j=1}^k|\nabla u_j|^2=1\quad\text{on}\quad\partial\{|U>0|\}\cap B_1,$$
into a condition of the form
$$|\nabla u_j|=g(x)\quad\text{on}\quad\partial\{|U>0|\}\cap B_1,$$
involving just one of the components of $U$ and an auxiliary H\"older continuous function $g:\partial\Omega\to\R$. In order to prove that the Boundary Harnack principle holds on $\Omega_U:=\{|U|>0\}$, in \cite{CSY} it was shown that $\Omega_U$ is an NTA domain, while in \cite{mtv} it was proved that $\Omega_U$ is Reifenberg-flat; in both cases the conclusion followed from \cite{JK}. In this case \cref{t:main} offers an alternative approach. In fact, the modulus $|U|$ of a variational solution $U$ satisfies the conditions of \cref{t:main}. In fact, \ref{item:1-positivity} and \ref{item:4-subharmonicity} are clearly satisfied. For the Lipschitz continuity \ref{item:2-lipschitz} and the non-degeneracy \ref{item:7-nondegeneracy} of $|U|$ we refer to \cite{mtv}, while \ref{item:6-levels} was proved in \cite[Section 2.2]{mtv2}.
Moreover, in the non-degenerate case, in \cite{krli} it was shown that up to a constant one can bound $|U|$ from above with $u_1$. Thus, \ref{item:3-distance} is an immediate consequence from the classical interior Harnack inequality and the non-degeneracy of $|U|$. Finally, the exterior density estimate \ref{item:5-density} was proved in \cite{mtv}.
%

\subsection{Subsolutions and supersolutions}\label{sub:inwards-otwards}
For every $\Lambda>0$ and every non-negative function $u:B_1\to\R$ we define the functional
$$\mathcal F_\Lambda(u):=\int_{B_1}|\nabla u|^2\,dx+\Lambda\big|\{u>0\}\big|.$$
We say that $u$ is a supersolution (subsolution) of $\mathcal F_\Lambda$, if
$$\mathcal F_\Lambda(u)\le \mathcal F_\Lambda(v),$$
for every non-negative $v:B_1\to\R$ with the same boundary data as $u$ and such that $u\le v$ ($u\ge v$) in $B_1$. It is easy to show that if $u$ is at the same time a sub- and a supersolution of $\mathcal F_\Lambda$, then $u$ is actually a minimizer of $\mathcal F_\Lambda$ and so, by the classical result of Alt and Caffarelli \cite{alca}, the free boundary $\partial\{u>0\}$ is smooth in $B_1$ up to a set of small Hausdorff dimension. On the other hand, if $u$ is a subsolution for some $\mathcal F_\lambda$ and a supersolution for some $\mathcal F_\Lambda$, then nothing is known about the local structure of the free boundary. Still, from the analysis in \cite[Sections 3 and 4]{alca} (see also \cite{velectures}), one can easily check that we have the following result.
\begin{proposition}\label{p:sub-super}
Suppose that $0<\lambda<\Lambda$ are two constants and that $u\in H^1(B_1)$ is a non-negative function, which is a subsolution for $\mathcal F_\lambda$ and a supersolution for $\mathcal F_\Lambda$. Then, $u$ satisfies the conditions \ref{item:1-positivity}, \ref{item:2-lipschitz}, \ref{item:3-distance}, \ref{item:4-subharmonicity}, \ref{item:5-density}, \ref{item:6-levels} and \ref{item:7-nondegeneracy} of \cref{t:main} and the Boundary Harnack principle holds on the set $\Omega_u=\{u>0\}\cap B_1$.
\end{proposition}

\end{document}